\documentclass[3p,sort&compress]{personalartcl}
\usepackage[T1]{fontenc}

\usepackage{enumerate}
\usepackage{enumitem}
\usepackage[british]{babel}
\usepackage[utf8]{inputenc}
\usepackage{dsfont,bbm,mathrsfs}
\usepackage{amsmath, amsthm, amsbsy, amssymb}
\usepackage[all]{xy}
\usepackage{tikz}
\usetikzlibrary{matrix,arrows,chains,positioning,scopes,decorations.pathmorphing,backgrounds,fit}

\newcommand{\remove}[1]{}

\renewcommand{\leq}{\leqslant}
\renewcommand{\geq}{\geqslant}

\DeclareMathOperator{\Z}{\mathbb{Z}}

\DeclareMathOperator{\R}{\mathbb{R}}

\newcommand{\ie}{\textit{i.e.}}
\newcommand{\cf}{\textit{cf.}}
\newcommand{\eg}{\textit{e.g.}}

\theoremstyle{theorem}

\newtheorem{theorem}{Theorem}[section]
\newtheorem*{theorem1}{Theorem}

\newtheorem{lemma}[theorem]{Lemma}
\newtheorem{lemmac}[theorem]{Chinese Remainder Theorem}
\newtheorem{lemmap}[theorem]{The Patching Lemma}
\newtheorem{lemmas}[theorem]{The strongly semisimple Patching Lemma}
\newtheorem{lemmasc}[theorem]{The strongly semisimple Chinese Remainder Theorem}
\newtheorem{corollary}[theorem]{Corollary}

\theoremstyle{definition}

\newtheorem{remark}[theorem]{Remark}

\newtheorem{example}[theorem]{Example}

\DeclareMathOperator{\Spec}{\rm Spec}
\DeclareMathOperator{\Max}{\rm Max}
\DeclareMathOperator{\Rad}{\rm Rad}
\newcommand{\p}{\mathfrak{p}}

\newcommand{\m}{\mathfrak{m}}
\DeclareMathOperator{\V}{\mathbb{V}}
\DeclareMathOperator{\I}{\mathbb{I}}

\DeclareMathOperator{\Idl}{\rm Idl}
\DeclareMathOperator{\PrIdl}{\rm Prin}
\DeclareMathOperator{\C}{\rm C}

\DeclareMathOperator{\cl}{\rm cl}

\begin{document}
\begin{frontmatter}
\title{The Chinese Remainder Theorem  for strongly semisimple MV-algebras and lattice-groups.\\{\medskip\small {\hfill \sl For Tonino}}}
\author{Vincenzo Marra}
\ead{vincenzo.marra@unimi.it}
\address{Dipartimento di Matematica ``Federigo Enriques''\\ Universit\`a degli
Studi di Milano\\ Via Cesare Saldini 50, 20133 Milano, Italy}

\journal{Mathematica Slovaca, Antonio Di Nola's Festschrift,}
\date{May 25, 2013}
\begin{abstract}An MV-algebra (equivalently, a lattice-ordered Abelian group with a distinguished order unit) is strongly semisimple if all of its quotients modulo finitely generated congruences are semisimple. All MV-algebras satisfy a Chinese Reminder Theorem, as was first shown by Keimel four decades ago in the context of lattice-groups. In this note we prove that  the Chinese Remainder Theorem admits a considerable strengthening for strongly semisimple structures.
 \end{abstract}
\end{frontmatter}
\section{Introduction.}\label{s:intro}
We assume  familiarity with MV-algebras \cite{cdm} and lattice-ordered Abelian groups \cite{bkw}, which we write in additive notation and call \emph{$\ell$-groups} for short.  The mathematics in this note can be presented either in the language of MV-algebras, or in the categorically equivalent language of  $\ell$-groups with a strong order unit, thanks to  \cite[Theorem 3.9]{Mundici86}. We choose $\ell$-groups because we are going to use \cite[Chapitre 10]{bkw} as a convenient source for Keimel's spectral spaces of lattice-groups. We write `$\ell$-ideal' to mean `order-convex sublattice subgroup'; such are precisely the kernels of lattice-group homomorphisms, which are in bijection with  congruences.  If $I$ is an $\ell$-ideal of $G$ and $g \in G$, we write $G/J$ for the quotient $\ell$-group, and  $[g]_{I}\in G/I$ for the congruence class of $g$ modulo $I$, \ie\ for the coset $g/I:=g+I:=\{g +h \mid h \in I\}$. 
  It is clear that an arbitrary intersection of $\ell$-ideals is again an $\ell$-ideal. The $\ell$-ideal generated by a subset $S\subseteq G$ of an $\ell$-group $G$ is therefore defined as the intersection of all $\ell$-ideals of $G$ containing $S$.

\smallskip In his work on sheaf-theoretic representation of $\ell$-groups \cite{keimel}, \cite[Chapitre 10]{bkw}, Keimel used the following
\begin{lemmac}\label{l:cry}Let $\{I_i\}_{i=1}^n$ be a finite collection of $\ell$-ideals of an $\ell$-group $G$, and suppose that $\{g_i\}_{i=1}^n$ is a finite subset of $G$ such that 
\begin{align*}
[g_i]_{I_i \vee I_j}=[g_j]_{I_i \vee I_j}
\end{align*}
for each $i ,j \in \{1,\ldots, n\}$, where $I_i\vee I_j$ denotes the $\ell$-ideal of $G$ generated by $I_i \cup I_j$. Then there exists an element $g \in G$
 such that $[g]_{I_i}=[g_i]_{I_i}$ for each $i \in \{1,\ldots,n\}$.
\end{lemmac}
\begin{proof}This  is \cite[10.6.3]{bkw}.
\end{proof}
In one form or another, this result plays a  r\^{o}le
in any sheaf-theoretic representation of MV-algebras and lattice-groups. In Yang's Ph.D.\ Thesis \cite[Proposition 5.1.2, see also Remark 5.3.12]{Yang06}, Keimel's Chinese Remainder Theorem is used.
Dubuc's and Poveda's Pullback-Pushout Lemma \cite[3.11]{dubucpoveda}  is a consequence of the Chinese Remainder Theorem. In Filipoiu's and Georgescu's paper \cite{FilGeo1995} the result, while not featuring explicitly,  floats beneath the surface---see \eg\ \cite[Proposition 2.16]{FilGeo1995}. Schwartz \cite[Theorem 3.1]{Schwartz2013} invokes Keimel's result to obtain his  representation; on the other hand, Rump and Yang \cite{RuYa2008} use Keimel's sheaf-theoretic representation without employing the result directly. In \cite[Theorem 2.6]{FerrLett2011},  Ferraioli and Lettieri prove the  Chinese Remainder Theorem for MV-algebras. They do not use nor mention Keimel's  result; \cf\ also  \cite{DuPo2012}.  Let us also point out that sheaf-theoretic representations and Chinese Remainder Theorems were  studied at the level of general algebra by Vaggione in \cite{Vag1992},  whose results extend  the previous ones by Krauss and Clark \cite{KraCla1979}; see also Cornish's earlier paper \cite{Cor1977} in the same direction. For a  unified approach to sheaf-theoretic representations of MV-algebras using Stone-Priestley duality for distributive lattices, see \cite{gvm}.

\smallskip Here we prove that strongly semisimple unital $\ell$-groups and MV-algebras  enjoy a considerably stronger version of the general Chinese Remainder Theorem.  Let us recall some basic facts first. 

\smallskip
An $\ell$-ideal $\p$ of an $\ell$-group $G$ is \emph{prime}  \cite[2.4.1]{bkw} if it is proper (\ie\ $\neq G$), and the quotient $G/\p$ is totally ordered; and it is \emph{maximal} if it is proper, and maximal with respect to  inclusion. It is elementary that each maximal $\ell$-ideal is prime.  Principal (=singly generated)  and finitely generated $\ell$-ideals  can be shown to coincide.

 A \emph{unital} $\ell$-group $(G,u)$ is an $\ell$-group equipped with a \emph{\textup{(}strong order\textup{)} unit} \cite[2.2.12]{bkw}: a non-negative  element $u \in G$ whose positive multiples eventually dominate each element of $G$. The unital setting is convenient because it ensures that spaces of prime congruences are compact, see Section \ref{s:pre}. 
 By a \emph{unital $\ell$-subgroup} of the unital $\ell$-group $(G,u)$ we mean an $\ell$-subgroup of $G$ that contains $u$. By `$\ell$-homomorphism' we mean `lattice-group homomorphism', and by `unital $\ell$-homomorphism' we mean `unit-preserving $\ell$-homomorphism'.

The Archimedean property plays a substantive r\^{o}le in this paper.
 The unital $\ell$-group $(G,u)$ is \emph{simple} if it has no proper, non-trivial (\ie\ $\neq \{0\}$) $\ell$-ideals. It  is \emph{semisimple} if the intersection of all its maximal $\ell$-ideals is the trivial $\ell$-ideal $\{0\}$; equivalently, if it is a unital subdirect product of simple unital $\ell$-groups; equivalently,  if it is \emph{Archimedean}: for all $g,h \in G$ with $0< g\leq h$, there is an integer $m\geq 0$ such that $mg \not \leq h$. Not all of these formulations of the Archimedean property are equivalent in the absence of a unit; see  Lemma \ref{l:ss} and Remark \ref{r:ss} below. Finally, $(G,u)$ is \emph{strongly semisimple} if each one of its quotients modulo a principal  $\ell$-ideal is semisimple. Since we are not excluding the trivial principal $\ell$-ideal $\{0\}$ in the preceding definition, if $G$ is strongly semisimple then it is semisimple.
\begin{remark}In the literature on MV-algebras, the term ``strongly semisimple''  first appeared in \cite{dubucpoveda}. However, the closely related concept of ``logically complete MV-algebras'' was already introduced by Belluce and Di Nola in \cite{BellDiN}. By contrast, to the best of our knowledge strong semisimplicity has not been considered  in the literature on lattice-groups.\qed
\end{remark}
 Now let us consider  a unital $\ell$-group $(G,u)$ with unit $u$, and let $\Max{G}$ denote its set of maximal $\ell$-ideals. The subsets  of the form $\{\m \in \Max{G} \mid \m \supseteq I\}$, as $I$ ranges over all possible $\ell$-ideals of $G$, are the closed sets of a compact Hausdorff topology on $\Max{G}$, known as the \emph{hull-kernel topology}, \cf\ \cite[p.\ 111]{gj}. A  construction of Yosida \cite{yosida}, significantly generalised in \cite{hager_robertson}, associates to each element $g \in G$ a  continuous function $\widehat{g}\colon \Max{G}\to \R$. In the  unital case, the crux of the matter is the classical result by H\"older \cite{holder} that each quotient $(G/\m, [u]_\m)$, with $\m$ a maximal $\ell$-ideal of $G$, admits a \emph{unique} unit-preserving injective  $\ell$-homomorphism $\mathfrak{h}_{\m}\colon (G/\m, [u]_\m) \hookrightarrow (\R,1)$. Then $\widehat{g}(\m)$ is the unique real number $\mathfrak{h}_{\m}([g]_\m)$. See  Lemma \ref{l:holder} below. A subset $Z\subseteq \Max{G}$ is a \emph{principal zero set} if it is of the form $Z=\widehat{g}^{-1}(0)$ for some $g\in G$; equivalently, $Z=\{\m \in \Max{G} \mid g \in \m\}$. It is an exercise to check that the principal zero sets form a basis of closed sets for the hull-kernel topology of $\Max{G}$. Main result:
  \begin{theorem1}Let $(G,u)$ be a unital $\ell$-group, and let $\{Z_i\}_{i=1}^n$ be a finite collection of  principal zero sets of $\Max{G}$. Suppose that $\{g_i\}_{i=1}^n\subseteq G$ is compatible with $\{Z_i\}_{i=1}^n$, meaning  that $\widehat{g}_i(\m)=\widehat{g}_j(\m)$ for each $i, j \in \{1,\ldots, n\}$ and each $\m \in Z_i \cap Z_j$. If $(G,u)$ is strongly semisimple, there exists an element $g \in G$ such that $\widehat{g}(\m)=\widehat{g}_i(\m)$ for each $i \in \{1,\ldots,n\}$ and each $\m
\in Z_i$. If, moreover, $\bigcup_{i=1}^n Z_i=\Max{G}$, such an element $g$ is unique.
\end{theorem1}
We prove the theorem in Section \ref{s:theo},  after some preliminaries in Section \ref{s:pre}. We will see in due course that the theorem above is a Chinese Remainder Theorem for semisimple structures, stated in the functional language of the Yosida representation. Indeed, while the result is formulated here  in terms
of $\Max{G}$,   we prove as Lemma \ref{d:sspatch} an equivalent statement for the  spectral space of \emph{prime} $\ell$-ideals of $G$. Lemma \ref{d:sspatch} is a (new) strengthening for strongly semisimple structures of the (known)  Patching Lemma \ref{c:patch} that holds in full generality. The Patching Lemma, in turn, is the dual spectral version of Keimel's Chinese Remainder Theorem  \ref{l:cry} for lattice-groups. Symmetrically, the dual algebraic version of Lemma \ref{d:sspatch}, given below as  Theorem \ref{sscrt}, provides a  (new) Chinese Remainder Theorem for strongly semisimple structures. Our reformulation of the theorem above in the spectral language of  Lemma \ref{d:sspatch} requires the characterisation of strongly semisimple structures in terms of the topology of the prime spectrum. This we achieve in Corollary \ref{c:ss}; a similar result for Archimedean $\ell$-groups has been previously obtained by Yang \cite[Proposition 5.3.1]{Yang06}.

\section{Preliminary results.}\label{s:pre}
Throughout this note we let $(G,u)$ denote a unital $\ell$-group.
\subsection{Spectral spaces of unital lattice-ordered Abelian groups.}\label{ss:spec}
Consider the family $\Idl{G}$ of all $\ell$-ideals of $G$. Then, as proved in \cite[2.2.7 and 2.2.9]{bkw}, $\Idl{C}$ is a complete bounded distributive lattice under the operations $I\wedge J := I \cap J$ and $I\vee J:=\langle I \cup J\rangle$, where $I,J \in \Idl{G}$ and $\langle S \rangle$ denotes  the intersection of all $\ell$-ideals of $G$ containing $S$.
  Let $\PrIdl{G}$ denote the subset of
$\Idl{G}$ consisting of principal $\ell$-ideals. Then  $\PrIdl{G}$ is in fact a sublattice of $\Idl{G}$, \cite[2.2.11]{bkw}. Write $\Spec{G}$ for the set of prime $\ell$-ideals of $G$, and $\Max{G}\subseteq \Spec{G}$ for the set of all maximal $\ell$-ideals of $G$. 

For any $R\subseteq G$, set
\begin{align}\label{eq:v}
\V{(R)}:=\left\{\,\p \in \Spec{G}\, \mid \,\p \,\supseteq R\,\right\}\,.
\end{align}
Symmetrically, for any $S\subseteq \Spec{G}$, set
\begin{align}\label{eq:i}
\I{(S)}:=\left\{\,g \in G\, \mid \, g \in \p \text{ for each } \p \in S\,\right\}=\bigcap S\,.
\end{align}
For any set $X$, we let $2^{X}$ denote the set of the subsets of $X$.
\begin{lemma}[Keimel's Spectral Topology]\label{l:spec basics}Consider the functions $\V\colon 2^{G}\longrightarrow  2^{\Spec{G}}$ and $\I\colon 2^{\Spec{G}}\longrightarrow 2^{G}$
given by \textup{(\ref{eq:v}--\ref{eq:i})}.
\begin{enumerate}
\item The pair $(\V,\I)$ yields a \textup{(}contravariant\textup{)} Galois connection between the $2^{G}$ and $2^{\Spec{G}}$. That is, 
\begin{align*}
R\subseteq \I{(S)} \ \ \ \text{if, and only if,} \ \ \ S\subseteq   \V{(R)}
\end{align*}
for all $R\subseteq G$ and $S\subseteq \Spec{G}$.
\item $\V$ reverses arbitrary unions to intersections: $\V{\left(\bigcup_{i \in D} R_i\right)} = \bigcap_{i \in D} \V{\left(R_i\right)}$, where $R_i$ is a subset of $G$, and $D$ is an arbitrary index set. Similarly, $\I$ reverses arbitrary unions to intersections: $\I{\left(\bigcup_{i \in D} S_i\right)} = \bigcap_{i \in D} \I{\left(S_i\right)}$, where $S_i$ is a subset of $\Spec{G}$, and $D$ is an arbitrary index set.
\item The fixed subsets of $G$, \ie\ the ones for which $\I{(\V{(R)})}=R$, are precisely its $\ell$-ideals. 
\item The fixed subsets of $\Spec{G}$, \ie\ the ones for which $\V{(\I{(S)})}=S$, are precisely the \emph{vanishing loci} 
of $\ell$-ideals, that is, those of the form
\begin{align*}\tag{*}\label{t:zariskiclosed}
\V{(I)}, \ \ \text{for some } I \in \Idl{G}\,.
\end{align*}
\item The sets of the form \textup{(\ref{t:zariskiclosed})} are the closed sets of a topology\footnote{This topology on $\Spec{G}$ is variously called its \emph{Zariski}, or \emph{spectral}, or \emph{hull-kernel} topology.}  on $\Spec{G}$, whose associated closure operator is given by $\cl:=\V\circ\I\colon 2^{\Spec{G}}\longrightarrow  2^{\Spec{G}}$.
\item The space $\Spec{G}$ is \emph{spectral:}\footnote{In the sense of Hochster \cite[p.\ 43]{hochster}.} it is $T_{0}$, compact, its compact open subsets form a basis of opens sets stable under finite  intersections, and it is \emph{sober} --- \ie\ every non-empty, closed subset  that cannot be written as the union of two proper closed subsets, has a dense point. 
\item The compact and open subsets of $\Spec{G}$ are precisely the complements of those of the form $\V{(P)}$, for $P$ a principal $\ell$-ideal of $G$. Hence, the lattice $\PrIdl{G}$ is isomorphic and anti-isomorphic to the lattices of compact open and  co-compact\footnote{Throughout, we write `co-compact set' to mean `complement of a compact set'.} 
 closed subsets of $\Spec{G}$, respectively.\footnote{In particular, the  space prime lattice ideals of the bounded distributive lattice $\PrIdl{G}$, endowed with the usual Stone topology --- see \eg\ \cite{johnstone} ---  is homeomorphic to $\Spec{G}$.} 
 \item Equip  $\Max{G}$ with the subspace topology that it inherits from $\Spec{G}$. Then $\Max{G}$  is a compact Hausdorff space. 
\end{enumerate}
\end{lemma}
\begin{proof}Item 1 is an exercise. Item 2 is a general property of Galois connections, see \eg\ \cite{erneetal} for background. Items 3 and 4 are easy to prove using the fact that each $\ell$-ideal is an intersection of the prime $\ell$-ideals that contain it \cite[2.5.5]{bkw}. Items 5--7 are proved in \cite[10.1.2--10.1.10]{bkw}. Item 8 is \cite[10.2.5]{bkw}.
\end{proof}
It is immediate to describe  the closure operator $\cl\colon 2^{\Spec{G}}\longrightarrow  2^{\Spec{G}}$ in terms of intersections of primes. 
For any subset $S\subseteq \Spec{G}$, we have
\begin{align}\label{eq:cl}
\cl{S}=\left\{\p \in \Spec{G} \, \mid\, \p \supseteq \bigcap S \right\}.
\end{align}
Indeed, by definition $\cl{S}=\V{(\I{(S)})}=\{\p \in \Spec{G} \, \mid\, \p \supseteq \I{(S)}\}$, and $\I{(S)}=\bigcap S$.

\subsection{The Patching Lemma.}\label{s:crt}

\begin{lemmap}\label{c:patch}Let $\{U_i\}_{i=1}^n$ be a finite collection of  closed subsets of $\Spec{G}$, and suppose that $\{g_i\}_{i=1}^n$ is a finite subset of $G$ such that $[g_i]_\p=[g_j]_\p$ for each $i , j \in \{1,\ldots, n\}$ and each $\p \in U_i \cap U_j$. Then there exists an element $g \in G$, such that $[g]_\p=[g_i]_\p$ for each $i \in \{1,\ldots,n\}$ and each $\p \in U_i$. 
\end{lemmap}
\begin{proof}This is a direct dualisation of  the Chinese Remainder Theorem \ref{l:cry} to $\Spec{G}$, using Lemma \ref{l:spec basics}.
\end{proof}
\subsection{The Archimedean property,  semisimplicity, and the maximal spectrum.}
For $X$ a topological space, a subset $S\subseteq \C{(X)}$ is said to \emph{separate the points} (of $X$) if for each $x\neq y \in X$ there is $f\in S$ with $f(x)=0$ and $f(y)\neq 0$.
\begin{lemma}[H\"older's Theorem \cite{holder}, and Yosida's Theorem\footnote{While \cite[Theorems 1 and 2]{yosida} prove more than our (iii--iv), that is all we need in this paper.} \cite{yosida}]\label{l:holder} Let $(G,u)$ be a unital $\ell$-group.
\begin{enumerate}
\item If  $G$ is Archimedean and totally ordered, then there is a unique unital injective $\ell$-homomorphism $(G,u)\hookrightarrow (\R,1)$. \
\item If  $\m \in \Max{G}$, then there is a unique unital injective $\ell$-homomorphism $\mathfrak{h}_{\m}\colon (G/\m,u/\m)\hookrightarrow (\R,1)$.
\item If $g\in G$, the function 
\begin{align*}
\widehat{g}\colon \Max{G} &\longrightarrow \R\\
\m \in \Max{G} &\longmapsto \mathfrak{h}_{\m}(g/\m)\in\R,
\end{align*}
where $\mathfrak{h}_\m$ is given by \textup{2}, is continuous with respect to  the Euclidean topology of $\R$.
\item Let $X$ be a compact Hausdorff space, and assume that $G\subseteq \C{(X)}$ is an $\ell$-subgroup, and that  $u=1_X$. If $G$ separates the points of $X$, then $\Max{G}$ is homeomorphic to $X$ via the map
 \begin{align*}
 \m \in \Max{G} \ \longmapsto \ x(\m)\in X,
 \end{align*}
where $x(\m)$ is the unique point of $X$ in the set $\bigcap \{f^{-1}(0)\mid f \in \m\}$.
\end{enumerate}
\end{lemma}
\begin{proof} 1.
This is \cite[2.6.3]{bkw}, with the additional observation that the unital assumption makes the embedding unique.

\smallskip
\noindent 2. Observe that $(G/\m,[u]_\m)$ is simple, because $\m$ is maximal, and the $\ell$-ideals of $G/\m$ are in one-one inclusion-preserving correspondence with the $\ell$-ideals of $G$ containing $\m$  \cite[2.3.8]{bkw}. Simplicity  entails easily that $G/\m$ is Archimedean and totally ordered, so 1 applies.

\smallskip \noindent 3. This was  first proved in \cite[]{yosida}, for unital lattice-ordered vector spaces (known as \emph{vector lattices}). The proof for $\ell$-groups is a straightforward adaptation of Yosida's argument in \cite{yosida}, using H\"older's Theorem (1--2) in place of his \cite[Lemma 2]{yosida}.

\smallskip \noindent 4. Essentially  \cite[Theorem 4]{yosida}.
\end{proof}
\begin{lemma}\label{l:ss}For any unital $\ell$-group $(G,u)$, the following are equivalent.
\begin{enumerate}
\item $(G,u)$  is semisimple.
\item $G$ is Archimedean.
\item The \emph{radical} of $G$ is trivial, that is, $\Rad{G}:=\bigcap \Max{G}=\{0\}$.
\end{enumerate}
\end{lemma}
\begin{proof}The equivalence of 1 and 2 is an immediate consequence of the fact that a quotient $G/I$, for $I$ a proper $\ell$-ideal of $G$, is simple if, and only if, $I$ is maximal. And this, in turn, follows at once from \cite[2.3.8]{bkw}. The equivalence of 2 and 3 was first proved, for unital vector lattices, in \cite[Theorem 1]{yosidafuka}. The proof for unital $\ell$-groups is a straightforward adaptation using H\"older's Theorem (Lemma \ref{l:holder}.(1--2)) in place of \cite[Lemma 1]{yosidafuka}. The present lemma is also proved for MV-algebras in \cite[3.6.1 and 3.6.4]{cdm}.\end{proof}
\begin{remark}\label{r:ss}Caution: the equivalence of 2 and 3 in Lemma \ref{l:ss} fails in the absence of a unit. See \cite[\S 3]{yosidafuka} for an early example due to Nakayama.\qed
\end{remark}

\section{Proof of  theorem.}\label{s:theo}
%
%
%
We begin by adding to Lemma \ref{l:ss} a spectral characterisation of the Archimedean property. 
\begin{remark}The following lemma should be compared to Yang's aforementioned result \cite[Proposition 5.3.1]{Yang06}.\qed
\end{remark}

\begin{lemma}\label{l:ssspec}For any unital $\ell$-group $(G,u)$, the following are equivalent.
\begin{enumerate}
\item $(G,u)$  is semisimple.
\item $\Max{G}$ is dense in $\Spec{G}$.
\end{enumerate}
\end{lemma}
\begin{proof} First suppose that $G$ is not semisimple, hence not Archimedean by Lemma \ref{l:ss}. There exist $g,h \in G$ with $0<ng\leq h$ for all integers $n\geq 1$. Then $\V{(g)}\supseteq \Max{G}$. Indeed, assume by way of contradiction that $[g]_{\m}>0$ for some $\m \in \Max{G}$. By Lemma \ref{l:holder}.2, $[g]_{\m}, [h]_{\m}\in \R$ to within the unique unital $\ell$-embedding $\mathfrak{h}_{\m}$, and $0<[g]_{\m}\leq [h]_{\m}$. By the Archimedean property of $\R$, there is an integer $n_{0}\geq 1$ with $n_{0}[g]_{\m}\geq [h]_{\m}$, whence $n_{0}g \not \leq h$ in $G$, contradiction. This proves $\V{(g)}\supseteq \Max{G}$. Now, since $\V{(g)}$ is closed, the closure of $\Max{G}$ is contained in $\V{(g)}$.
And since $g \neq 0$ and $\bigcap \Spec{G}=\{0\}$ \cite[10.1.8]{bkw}, there is $\p \in G$ with $[g]_{\p}\neq 0$. In other words, $\V{(g)}\subset \Spec{G}$. In conclusion, $\Spec{G} \supset \V{(g)}\supseteq \cl{\Max{G}}$, and $\Max{G}$ is not dense in $\Spec{G}$. 

Conversely, suppose $\cl{\Max{G}}\subset \Spec{G}$. The compact open sets $\{\V{(g)}^{\rm c} \mid g \in f\}$, where $\V{(g)}^{\rm c}$ is the complement of  $\V{(g)}$ in $\Spec{G}$, form a basis of $\Spec{G}$, by Lemma \ref{l:spec basics}. Hence there is $S \subseteq G$ such that the  open set $O:=\left(\cl{\Max{G}}\right)^{\rm c}$ may be written as $O=\bigcup_{g \in S} \V{(g)}^{\rm c}$. Therefore, $\V{(g)}\supseteq \Max{G}$ for each $g \in S$. Since $O$ is non-empty, we
must have $\V{(g_{0})}^{\rm c}\neq\emptyset$ for some $g_0 \in S$, so that $g_{0}\neq 0$.  But then, since $g_{0}\in\m$ for each $\m\in\Max{G}$,
we have $0\neq g_{0}\in \Rad{G}=\bigcap\Max{G}\neq \{0\}$. By Lemma \ref{l:ss}, $G$ is not semisimple.
\end{proof}
\begin{lemma}\label{l:specquotients}Let $(G,u)$ be a unital $\ell$-group, and let $I$ be an $\ell$-ideal of $G$. Then $\Spec{(G/I)}$ is homeomorphic to $\V{(I)}$, equipped with the subspace topology it inherits from $\Spec{G}$.
\end{lemma}
\begin{proof}By \cite[2.3.8]{bkw}, there is a one-one inclusion preserving correspondence between the $\ell$-ideals of $G$ containing $I$, and the $\ell$-ideals of $G/I$. Specifically, the correspondence is given by
\begin{align*}
J \supseteq I \ \overset{\rho}{\longmapsto} \ J/I:=\{[j]_I \mid j \in J \}.
\end{align*}
The bijection $\rho$, being inclusion-preserving, restricts to a bijection between maximal $\ell$-ideals of $G$ containing $I$, and maximal $\ell$-ideals of $G/I$. To see that $\rho$ also restricts to a bijection between prime $\ell$-ideals, we use the result in universal algebra \cite[3.11]{cohn} that $G/J\cong (G/I)/(J/I)\cong
(G/I)/\rho(J)$. Since $G/J$ is totally ordered if, and only if, $J$ is prime, $\rho(J)$ is prime if, and only if, $J$ is prime. Since $\V{(I)}$ consists of the prime $\ell$-ideals of $G$ extending $I$, it is clear by the same token that the subspace topology on $\V{(I)}$ agrees with the topology on $\Spec{G/I}$.
\end{proof}
\begin{corollary}[Spectral characterisation of strong semisimplicity]\label{c:ss}Let $(G,u)$ be a unital $\ell$-group. Then
$(G,u)$ is strongly semisimple if, and only if, each closed, co-compact subset $K\subseteq \Spec{G}$ satisfies $K=\cl{(K\cap \Max{G})}$.
\end{corollary}
\begin{proof}Combine Lemmata \ref{l:ssspec} and \ref{l:specquotients}.
\end{proof}

\begin{lemmas}\label{d:sspatch}Let $(G,u)$ be a unital $\ell$-group, and let $\{U_i\}_{i=1}^n$ be a finite collection of  closed, co-compact subsets of $\Spec{G}$. Suppose that $\{g_i\}_{i=1}^n\subseteq G$ is such that $[g_i]_\m=[g_j]_\m$ for each $i , j \in \{1,\ldots, n\}$ and each $\m \in U_i \cap U_j \cap \Max{G}$. If $(G,u)$ is strongly semisimple, there exists an element $g \in G$ such that $[g]_\p=[g_i]_\p$ for each $i \in \{1,\ldots,n\}$ and each $\p \in U_i$.
\end{lemmas}
\begin{proof}
We set
\begin{align*}
M_i &:= U_i \cap \Max{G} \textup{ for each } i\in\{1,\ldots, n\}.
\end{align*}
Let $\p \in U_i \cap U_j$ for some $i,j \in \{1,\ldots,n\}$. If we show  $[g_i]_\p=[g_j]_\p$, then the Patching Lemma \ref{c:patch} applies, and the present lemma follows.
Equivalently,
upon setting $h:=g_i-g_j$, we need to show $[h]_{\p}\neq 0$.

Since closed, co-compact sets are stable under finite intersections (Lemma \ref{l:spec basics}.7), $K:=U_i\cap U_j$ is closed and co-compact.  Since $(G,u)$ is strongly semisimple, by Corollary \ref{c:ss} we have
\begin{align}\label{e:density}
K=\cl{(K\cap \Max{G})}.
\end{align}
If we had $[h]_{\m}=0$ for each $\m \in K\cap \Max{G}
$, we would infer $[h]_{\p}=0$, too. Indeed, since $\p \in \cl{(K\cap \Max{G})}$ by (\ref{e:density}), we  have $\p \supseteq  \bigcap (K\cap \Max{G})$ by (\ref{eq:cl}). Also, since $h\in\m$ for each $\m \in K\cap \Max{G}$, we have $h \in \bigcap (K\cap \Max{G})$. Hence $h \in \p$, that is, $[h]_\p=0$, and the proof is complete. 
\end{proof}
\paragraph{End of proof of theorem} Let $\{h_i\}_{i=1}^n\subseteq G$ be elements such that $Z_i=\{\m \in \Max{G} \mid h_i\in \m\}$, $i=1,\ldots,n$. Set
\begin{align*}
U_i := \V{(h_i)}, \, i=1,\ldots,n,
\end{align*}
so that $U_i \cap \Max{G}=Z_i$.
By hypothesis,  $\widehat{g}_i(\m)=\widehat{g}_j(\m)$ for each $i, j \in \{1,\ldots, n\}$ and each $\m \in Z_i \cap Z_j$,  \ie\ $[g_i]_\m=[g_j]_\m$ for each $i, j \in \{1,\ldots, n\}$ and each $\m \in U_i \cap U_j\cap \Max{G}$. An application of Lemma \ref{d:sspatch} yields $g \in G$ such that $[g]_\p=[g_i]_\p$ for each $i \in \{1,\ldots,n\}$ and each $\p \in U_i$; in particular, $\widehat{g}(\m)=\widehat{g}_i(\m)$ for each $i \in \{1,\ldots,n\}$ and each $\m
\in Z_i$. To prove the uniqueness assertion, assume $\Max{G}=\bigcup_{i=1}^n Z_i$, ad suppose $g, g'\in G$ satisfy $[g]_\m=[g']_\m=[g_i]_\m$ for each $i\in \{1,\ldots,n\}$ and each $\m \in Z_i$. Then $[g]_\m=[g']_\m$ for each $\m \in \bigcup_{i=1}^n Z_i =\Max{G}$. Since $(G,u)$ is semisimple, we have $\Rad{G}=\{0\}$  by Lemma \ref{l:ss}. Then  the
map $G \longrightarrow \prod_{\m \in \Max{G}} G/\m$ given by $g \longmapsto ([g]_\m)_{\m \in \Max{G}}$ has trivial kernel, and thus is injective. It follows that $g=g'$, as was to be shown. \qed
\begin{lemmasc}\label{sscrt} Let $(G,u)$ be a  unital $\ell$-group, and let  $\{I_i\}_{i=1}^n$ be a finite collection of principal $\ell$-ideals of  $G$. Suppose that $\{g_i\}_{i=1}^n$ is a finite subset of $G$ such that 
\begin{align*}
[g_i]_{\m}=[g_j]_{\m}
\end{align*}
for each $i ,j \in \{1,\ldots, n\}$ and each maximal $\ell$-ideal $\m$ that extends $I_{i}\vee I_{j}$. If $(G,u)$ is strongly semisimple, there exists an element $g \in G$
 such that $[g]_{I_i}=[g_i]_{I_i}$ for each $i \in \{1,\ldots,n\}$.
\end{lemmasc}
\begin{proof}This is a direct dualisation of  the strongly semisimple Patching Lemma \ref{d:sspatch} to $G$, using Lemma \ref{l:spec basics}.
\end{proof}
\begin{example}Theorem \ref{sscrt} fails in the absence of strong semisimplicity, and hence so does the theorem stated in Section \ref{s:intro}. Consider, for example, the lexicographic product $\Z\vec{\times} \Z$ of the additive group of integers with itself. Thus, $\Z\vec{\times} \Z$ is free Abelian of rank $2$ as a group, and $(a,b)\geq (a',b')$ if either $a > a'$, or else $a=a'$ and then $b\geq b'$. Now $\Z\vec{\times}\Z$ is not  semisimple, and  has exactly two $\ell$-ideals: the  maximal $\ell$-ideal $I_1=:\{(0,b)\mid b \in \Z\}$, and the prime non-maximal $\ell$-ideal $I_2:=\{(0,0)\}\subseteq I_1$. They are evidently both principal. Consider the elements $g_1:=(0,0)$ and $g_2:=(0,1)$. Then $[g_1]_{I_1}=[g_2]_{I_1}$. However, there can be no $g\in\Z\vec{\times}\Z$ with $[g]_{I_2}=[g_1]_{I_2}=[g_2]_{I_2}$, simply because $[g_1]_{I_2}=g_1\neq g_2= [g_2]_{I_2}$.\qed
\end{example}

\end{document}